\documentclass[11pt]{article}
\usepackage{amscd}
\usepackage{amsfonts}
\usepackage{amsmath}
\usepackage{amssymb}
\usepackage{amsthm}
\usepackage{bbm}
\usepackage{CJK}
\usepackage{fancyhdr}
\usepackage{graphicx}
\usepackage{indentfirst}
\usepackage{latexsym,bm}
\usepackage{mathrsfs}
\usepackage{xypic}
\usepackage[square, comma, sort&compress, numbers]{natbib}
\usepackage[colorlinks,linkcolor=red, anchorcolor=blue]{hyperref}
\allowdisplaybreaks[4]
\newtheorem{theorem}{Theorem}[section]
\newtheorem{lemma}[theorem]{Lemma}
\newtheorem{definition}[theorem]{Definition}
\newtheorem{proposition}[theorem]{Proposition}

\newtheorem{corollary}[theorem]{Corollary}
\newtheorem{remark}[theorem]{Remark}

\usepackage[top=1in,bottom=1in,left=1.25in,right=1.25in]{geometry}
\def\<{\langle}
\def\>{\rangle}
\def\a{\alpha}
\def\b{\beta}

\def\d{\delta}

\def\di{\diamond}
\def\e{\eta}
\def\g{\gamma}

\def\l{\lambda}

\def\o{\otimes}
\def\p{\phi}
\def\r{\rho}

\date{}
\begin{document}
\renewcommand{\baselinestretch}{1.2}
\renewcommand{\arraystretch}{1.0}
\title{\bf Poisson Hopf module Fundamental theorem for Hopf group coalgebras}
 \date{}
\author {{\bf Daowei Lu$^1$\footnote {Corresponding author:  ludaowei620@126.com}, Dingguo Wang$^2$}\\
{\small $^1$School of Mathematics and Big data, Jining University}\\
{\small Qufu, Shandong 273155, P. R. China}\\
{\small $^2$School of Mathematical Sciences, Qufu Normal University}\\
{\small Qufu, Shandong 273165, P. R. China}
}
 \maketitle
\begin{center}
\begin{minipage}{12.cm}

\noindent{\bf Abstract.} Let $H$ be a Hopf group coalgebra with a bijective antipode and $A$ an $H$-comodule Poisson algebra. In this paper, we mainly generalize the fundamental theorem of Poisson Hopf modules to the case of Hopf group coalgebras. Finally we will deduce the relative projectivity in the category of Poisson Hopf modules.
\\

\noindent{\bf Keywords:} Hopf group coalgebras; Poisson algebra; Poisson module; Poisson Hopf module.
\\

 \noindent{\bf  Mathematics Subject Classification:} 17B63, 16T05.
 \end{minipage}
 \end{center}
 \normalsize\vskip1cm

\section*{Introduction}

Poisson algebras have lately been playing an important role in algebra~\cite{BLLM,GR}, geometry~\cite{BV}, mathematical
physics\cite{Kon,CP} and other subjects\cite{Od}. The study of Poisson algebras led to other algebraic structures, such as noncommutative Poisson algebras\cite{Xu}, Jacobi algebras (also called generalized Poisson algebras)\cite{Agore}, Gerstenhaber algebras\cite{Kos}, Lie-Rinehart algebras\cite{Rin} and Novikov-Poisson algebras\cite{XXu}. 

In the theory of Hopf algebras, the fundamental theorem of Hopf modules states that any Hopf module could be decomposed through its  coinvariant, which is an important result and  plays an essential role in the theory of Hopf algebras, such as the integral theory, Nichols-Zoller theorem, Galois theory and so on~\cite{Mon}.

Doi~\cite{D} introduced the concept of relative Hopf modules, which is a generalization
of Hopf modules. Moreover, he gave the fundamental theorem of relative Hopf modules, which is a weaker version of  the fundamental theorem of Hopf modules. Explicitly let $H$ be a Hopf algebra, and $A$ a right $H$-comodule algebra. If there exists a right $H$-comodule map $\phi:H\rightarrow A$ which is also an algebra map, then for any relative right $(A,H)$-Hopf module $M$, the following isomorphism of relative right $(A,H)$-Hopf modules is obtained
$$A\o_{A^{coH}}M^{coH}\rightarrow M,\ m\o a\mapsto m\cdot a.$$

Let $H$ be a Hopf algebra and $A$ a Poisson algebra such that $A$ is an $H$-comodule Poisson algebra. In~\cite{Gu} Gu$\acute{e}$d$\acute{e}$non established the fundamental theorem of Poisson $(A,H)$-Hopf modules, which generalized the fundamental isomorphism to the case of Poisson algebra. Following the research of Gu$\acute{e}$d$\acute{e}$non, in \cite{Lu} the authors generalized the fundamental isomorphism of Poisson Hopf modules to the case of weak Hopf algebras. 

Let $G$ be a group. V.G. Turaev\cite{Tur} introduced the notion of a modular crossed $G$-category and showed that such a category gives rise to a three-dimensional homotopy quantum field theory with target space $K(G,1)$. Meanwhile  plays a key role in the construction of Hennings-type invariants of flat group-bundles over complements of link in the 3-sphere\cite{V2}.  Examples of $G$-categories
can be constructed from the so-called Hopf $G$-coalgebras also intoduced in\cite{Tur}.  A. Virelizier\cite{V1} studied algebraic properties of Hopf $G$-coalgebras and given the fundamental theorem for Hopf $G$-comodules. Later Li and Chen generalized this classic result to the case of relative Hopf $G$-comodules\cite{Li}.

Motivated by these results, in this paper, we mainly generalize the fundamental theorem of Poisson Hopf modules to the case of Hopf group coalgebras. The paper is organized as follows. In section 1, we will recall basic results on Hopf group coalgebra and Poisson algebras. Let $H$ be a Hopf group coalgebra  and $A$ a family of Poisson algebras. In section 2, , we will firstly introduce the notion of Poisson $(A, H)$-Hopf modules and give the first main result:
\begin{theorem}
Let $H$ be a Hopf $G$-algebra and $A$  an $H$-comodule Poisson algebra. Suppose that there is an $H$-colinear map $\p:H\rightarrow A^A$ such that for all $\a\in G$, $\p_\a(1_{H_\a})=1_{A_\a}$.

(1) If $H$ is commutative or  $\phi$ is an algebra map, then every Poisson $(A, H)$-Hopf module which is injective as a Lie $A$-module is an injective $(\underline{A}, H)$-comodule.

(2) If $H$ is commutative and $\phi$ is an algebra map, then every Poisson $(A, H)$-Hopf module which is injective as a Poisson $A$-module is an injective Poisson $(A, H)$-Hopf module.
\end{theorem}

In section 3, we will establish the fundamental isomorphism of Poisson $(A,H)$-Hopf modules. That is,
\begin{theorem}
 Let $M$ be a Poisson $(A, H)$-Hopf module, and $\phi:H\rightarrow A^A$ a right $H$-colinear algebra map. Suppose $M^{coH}$ and $A^{coH}$ are trivial Lie $A_e$-modules under $\di'$. Then the morphism $\Phi$ given in Lemma \ref{lem:2k} is an isomorphism of Poisson $(A, H)$-Hopf modules.
\end{theorem}
In the end we will deduce relative projectivity in the category of Poisson $(A, H)$-Hopf modules.

\section{Preliminaries}
\def\theequation{1.\arabic{equation}}
\setcounter{equation} {0}

Throughout this article,  let $k$ be a fixed field, and all the vector spaces, tensor product and homomorphisms are over $k$. 

Let $G$ be an abelian group with the unit element $e$. Recall from \cite{Tur} that a $G$-coalgebra is a family $H=\{H_\a\}_{\a\in G}$ of $k$-spaces together with a family of $k$-linear maps $\Delta=\{\Delta_{\a,\b}:H_{\a\b}\rightarrow H_\a\o H_\b\}_{\a,\b\in G}$  (called a comultiplication) and a $k$-linear map $\varepsilon:H_e\rightarrow k$ (called a counit), such that $\Delta$ is coassociative and 
$$(id_\a\o\varepsilon)\Delta_{\a,e}=id_\a=(\varepsilon\o id_\a)\Delta_{e,\a}.$$

A Hopf $G$-coalgebra is a $G$-coalgebra $H=\{H_\a\}_{\a\in G}$ such that for all $\a,\b\in G$,
\begin{itemize}
  \item [(1)] each $H_\a$ is an algebra;
  \item [(2)] $\varepsilon$ and $\Delta_{\a,\b}$ are algebras maps;
  \item [(3)] there exists a family of $k$-linear maps $S=\{S_\a:H_\a\rightarrow H_{\a^{-1}}\}_{\a\in G}$ satisfying 
  $$m_{\a}(S_{\a^{-1}}\o id_{H_{\a}})\Delta_{\a^{-1},\a}=\varepsilon1_e=m_{\a^{-1}}(id_{H_{\a^{-1}}}\o S_{\a})\Delta_{\a^{-1},\a}.$$
\end{itemize}

Let $H$ be a Hopf $G$-coalgebra, then we have the following identities:
\begin{align*}
&S_\a(b)S_\a(a)=S_{\a}(ab),\ S_e(1_e)=1_e,\\
&\Delta_{\b^{-1},\a^{-1}}S_{\a\b}=\sigma_{H_{\a^{-1}},H_{\b^{-1}}}(S_\a\o S_\b)\Delta_{\a,\b},\ \varepsilon S_e=\varepsilon,
\end{align*}
for all $a\in H_\a,b\in H_\b$, where  $\sigma$ denotes the flip map.

Throughout this paper, we will use the Sweedler’s notation: for all $\a,\b\in G$ and $h\in H_{\a\b}$,
$$\Delta_{\a,\b}(h)=h_{(1,\a)}\o h_{(2,\b)}.$$

Let $H=\{H_\a\}_{\a\in G}$ be a Hopf $G$-coalgebra. An right $H$-comodule is a family of vector spaces $M=\{M_\a\}_{\a\in G}$ endowed with a family of $k$-linear maps $\r=\{\r_{\a,\b}:M_{\a\b}\rightarrow M_\a\o H_\b\}$ satisfying
\begin{align*}
&(id_\a\o\r_{\b\g})\circ\r_{\a,\b\g}=(\r_{\a,\b}\o id_\g)\circ \r_{\a\b,\g},\\
&(\varepsilon\o id_\a)\circ\r_{e,\a}=id_\a=( id_\a\o\varepsilon)\circ\r_{\a, e}.
\end{align*}
Similarly, we use the Sweedler’s notation for coactions
$$\r_{\a,\b}(m)=m_{(0,\a)}\o m_{(1,\b)},$$
for all $m\in M_{\a\b}$.

An $H$-comodule map between two right $H$-comodules $M$ and $N$ is a family of $k$-linear maps $f=\{f_\a:M_\a\rightarrow N_\a\}_{\a\in G}$ such that for all $m\in M_{\a\b}$,
$$
f_{\a\b}(m)_{(0,\a)}\o f_{\a\b}(m)_{(1,\b)}=f_\a(m_{(0,\a)})\o m_{(1,\b)}.
$$
Let $M$ be a right $H$-comodule. The coinvariants of $H$ on $M$ are the elements of the space
$$\left\{m=(m_\a)\in \prod\limits_{\a\in G}M_\a\Big{|}\r_{\a\b}(m_{\a\b})=m_\a\o 1_\b \hbox{\ for all }\a,\b\in G\right\}.$$
Let $M^{coH}_\a$ be the image of the (canonical) projection of this set onto $M_\a$. It is easy to verify that $M^{coH}=\{M^{coH}_\a\}_{\a\in G}$ is a right subcomodule of $M$, called the subcomodule of coinvariants.

Let $A$ be a family of Poisson algebras $A=\{A_\a\}_{\a\in G}$ (called group Poisson algebra), and denote the Poisson center of $A_\a$ by $A^{A_\a}_\a$ . The Poisson center of $A$ is given by 
$$A^A=\{A^{A_\a}_\a\}_{\a\in G},$$ 



Let $H$ be a Hopf $G$-coalgebra and $A$ a right $H$-comodule algebra. An  $(A,H)$-Hopf module is a right $H$-comodule $M$ with the action of $A_\a$ on $M_\a$ such that for all $\a,\b\in G, a\in A_{\a\b}, m\in M_{\a\b},$
$$(a\cdot m)_{(0,\a)}\o (a\cdot m)_{(1,\b)}=a_{(0,\a)}\cdot m_{(0,\b)}\o a_{(1,\b)} m_{(1,\b)}.$$

\section{The category $_{\mathcal{P}A}\mathcal{M}^H$}
\def\theequation{2.\arabic{equation}}
\setcounter{equation} {0}

\begin{definition}\label{Def:2a}
Let $A$ be a $G$-Poisson algebra. The family of spaces $M=\{M_\a\}_{\a\in G}$ is called a Poisson $A$-module if $(M,\cdot)$ is an $A$-module and $(M,\di)$  is a Lie $A$-module satisfying the following conditions
\begin{align}
  a\di(b\cdot m) &=\{a,a'\}\cdot m+b\cdot(a\di m),\label{2a} \\
  (ab)\di m &=a\cdot(b\di m)+b\cdot(a\di m),\label{2b}
\end{align}
for all $a, b\in A_{\a}, m\in M_\a.$
\end{definition}

We denote by $_{\mathcal{P}_A}\mathcal{M}$ the category of Poisson $A$-modules with the morphisms of $A$-modules and Lie $A$-modules.

\begin{definition}\label{Def:2b}
Let $A$ be a $G$-Poisson algebra and $H$ a Hopf $G$-coalgebra. $A$ is called a right $H$-comodule Poisson algebra if  $A$ is right $H$-comodule algebra and 
\begin{equation}
  \{a,b\}_{(0,\a)}\o \{a,b\}_{(1,\b)}=\{a_{(0,\a)},b_{(0,\a)}\}\o a_{(1,\b)}b_{(1,\b)},\label{2c}
\end{equation}
for all $a,b\in A_{\a\b}$.
\end{definition}

\begin{definition}\label{Def:2c}
Let $H$ be a Hopf  $G$-coalgebra and $A$ a right $H$-comodule Poisson algebra. The family of vector spaces $M=\{M_\a\}_{\a\in G}$ is called a Poisson $(A,H)$-Hopf module if $M$ is a $(A,H)$-Hopf module and a Poisson $A$-module such that  
\begin{equation}
(a\di m)_{(0,\a)}\o(a\di m)_{(1,\b)}=a_{(0,\a)}\di m_{(0,\a)}\o a_{(1,\b)} m_{(1,\b)},\label{2d}
\end{equation}
for all $a\in A_{\a\b}, m\in M_{\a\b}.$
\end{definition}

\begin{remark}
  It is obvious that the right $H$-comodule Poisson algebra $A$ itself is a Poisson $(A,H)$-Hopf module.
\end{remark}

We denote by $_{\mathcal{P}_A}\mathcal{M}^H$ the category of Poisson $(A,H)$-Hopf modules with the morphisms of $A$-modules, Lie $A$-modules and $H$-comodules.  For objects $M,N\in\! _{\mathcal{P}_A}\mathcal{M}^H$, we denote by $_{\mathcal{P}A}$Hom$^H(M,N)$ the vector space of  morphisms in $_{\mathcal{P}A}\mathcal{M}^H$.

For a Poisson $(A,H)$-Hopf module $M$, we denote$M^{AcoH}_\a=M^{coH}_\a\bigcap M^{A_\a}_\a$ and $M^{AcoH}=\left\{M^{AcoH}_\a\right\}_{\a\in G}.$

\begin{lemma}\label{Lem:2d}
Let $H$ be a Hopf $G$-coalgebra with a bijective antipode, $A$ a right $H$-comodule Poisson algebra and $M$ a Poisson $(A,H)$-Hopf module. Then 
\begin{itemize}
  \item [(1)] $M^A$ is a $H$-subcomodule of $M$.
  \item [(2)] $A^A$ is a $H$-subcomodule Poisson algebra of $A$.
  \item [(3)]  $A^{AcoH}$ is a $H$-subcomodule Poisson algebra of $A^A$.
  \item [(4)] $M^{AcoH}$ is a Poisson $A^{AcoH}$-submodule of $M$.
\end{itemize}
\end{lemma}

\begin{proof}
(1) For all $a\in A_\a, m\in M^{A_{\a\b}}_{\a\b},$
$$a\di m_{(0,\a)}\o m_{(1,\b)}=(1_e\o S^{-1}_{\b}(a_{(1,\b^{-1})}))\cdot[(a_{(0,\a\b)}\di m)_{(0,\a)}\o(a_{(0,\a\b)}\di m)_{(1,\b)}].$$
By a similar argument as in \cite[Lemma 1.5]{Gu}, we obtain $a\di m_{(0,\a)}=0$ for each summand $m_{(0,\a)}$. Hence $m_{(0,\a)}\in M^A_\a$. That is, $M^A$ is a $H$-subcomodule of $M$.

The rest of the proof is straightforward, analogue to the proof of \cite[Lemma 1.5]{Gu}.
\end{proof}

\begin{definition}\label{Def:2e}
Let $H$ be a Hopf $G$-coalgebra, and $A$ an $H$-comodule Poisson algebra. A family of vector spaces $M=\{M_\a\}_{\a\in G}$ is called an $(\underline{A},H)$-comodule if $M$ is a left Lie $A$-module, a right $H$-comodule and the relation (\ref{2d}) is satisfied.
\end{definition}

We denote by $_{\underline{A}}\mathcal{M}$ the category of Lie $A$-module with the morphisms of Lie $A$-modules, and $_{\underline{A}}\mathcal{M}^H$ the category of $(\underline{A},H)$-comodules with the morphisms of $A$-modules, Lie $A$-modules and $H$-comodules.

\begin{lemma}\label{Lem:2f}
  (1) Let $N=\{N_\a\}_{\a\in G}$ be a Lie $A$-module. Then 
$$N\o H=\left\{(N\o H)_{\a}\Big{|}(N\o H)_{\a}=\prod\limits_{\mu\nu=\a}N_\mu\o H_\nu\right\}_{\a\in G}$$ 
is an $(\underline{A},H)$-comodule with the $H$-coaction  and the Lie $A$-action given by
\begin{align*}
& (n\o h)_{(0,\mu')}\o  (n\o h)_{(1,\nu')} =n\o h_{(1,\mu^{-1}\mu')}\o h_{(2,\nu')},  \\
&a\di(n\o h)=a_{(0,\mu)}\di n\o a_{(1,\nu)}h,
\end{align*}
for all $a\in A_{\a},n\o h\in N_\mu\o H_\nu$ with $\mu\nu=\mu'\nu'=\a$.

(2) Furthermore if $N=\{N_\a\}_{\a\in G}$ be a Poisson $A$-module and $H$ is commutative, then $N\o H$ is a Poisson $(A,H)$-Hopf module with the $A$-action given by 
 $$a\cdot(n\o h)=a_{(0,\mu)}\cdot n\o a_{(1,\nu)}h,$$
for all $a\in A_{\a},n\o h\in N_\mu\o H_\nu$ with $\mu\nu=\a$.
\end{lemma}

\begin{proof}
  (1) For all $a,a'\in H_{\a},n\in N_\mu,h\in H_\nu$ with $\mu\nu=\a$,
\begin{align*}
  &\{a,a'\}\di(n\o h)\\
  &= \{a,a'\}_{(0,\mu)}\di n\o \{a,a'\}_{(1,\nu)}h\\
  &=\{a_{(0,\mu)},a'_{(0,\mu)}\}\di n\o a_{(1,\nu)}a'_{(1,\nu)}h\\
  &=\left[(a_{(0,\mu)}\di(a'_{(0,\mu)}\di n)-a'_{(0,\mu)}\di(a_{(0,\mu)}\di n)   \right]\o  a_{(1,\nu)}a'_{(1,\nu)}h\\
  &=a_{(0,\mu)}\di(a'_{(0,\nu)}\di n)\o a_{(1,\mu)}a'_{(1,\nu)}h-a'_{(0,\nu)}\di(a_{(0,\mu)}\di n)\o a_{(1,\mu)}a'_{(1,\nu)}h\\
  &=a\di(a'\di(n\o h))-a'\di(a\di(n\o h)),
\end{align*}
that is, $N\o H$ is a Lie $A$-module. And
\begin{align*}
&(a\di(n\o h))_{(0,\mu')}\o(a\di(n\o h))_{(1,\nu')}\\
&=(a_{(0,\mu)}\di n\o a_{(1,\nu)}h)_{(0,\mu')}\o (a_{(0,\mu)}\di n\o a_{(1,\nu)}h)_{(1,\nu')}\\
&=a_{(0,\mu)}\di n\o a_{(1,\mu^{-1}\mu')}h_{(1,\mu^{-1}\mu')}\o a_{(2,\nu')}h_{(2,\nu')}\\
&=a_{(0,\mu')}\di(n\o h_{(1,\mu^{-1}\mu')} )\o a_{(1,\nu')}h_{(2,\nu')}\\
&=a_{(0,\mu')}\di(n\o h)_{(0,\mu')}\o a_{(1,\nu')}(n\o h)_{(1,\nu')},
\end{align*}
where $\mu'\nu'=\a$. Therefore $N\o H$ is an $(\underline{A},H)$-comodule.

(2) Obviously $N\o H$ is an $A$-module.  For all $a,a'\in H_{\a},n\in N_\mu,h\in H_\nu$ with $\mu\nu=\a$,
\begin{align*}
&a\di(a'\cdot (n\o h))\\
&=a_{(0,\mu)}\di(a'_{(0,\mu)}\cdot n)\o a_{(1,\nu)}a'_{(1,\nu)}h\\
&=\left[\{a_{(0,\mu)},a'_{(0,\mu)}\}\cdot n+a'_{(0,\mu)}\cdot(a_{(0,\mu)}\di n)\right]\o a_{(1,\nu)}a'_{(1,\nu)}h\\
&=\{a_{(0,\mu)},a'_{(0,\mu)}\}\cdot n\o a_{(1,\nu)}a'_{(1,\nu)}h+a'_{(0,\mu)}\cdot(a_{(0,\mu)}\di n)\o a_{(1,\nu)}a'_{(1,\nu)}h\\
&=\{a,a'\}_{(0,\mu)}\cdot n\o \{a,a'\}_{(1,\nu)}h+a'_{(0,\mu)}\cdot(a_{(0,\mu)}\di n)\o a'_{(1,\nu)}a_{(1,\nu)}h\\
&=\{a,a'\}\cdot(n\o h)+a'\cdot(a\di (n\o h)),
\end{align*}
and 
\begin{align*}
&aa'\di (n\o h)\\
&=(aa')_{(0,\mu)}\di n\o (aa')_{(1,\nu)}h\\
&=a_{(0,\mu)}a'_{(0,\mu)}\di n\o a_{(1,\nu)}a'_{(1,\nu)}h\\
&=a_{(0,\mu)}\cdot(a'_{(0,\mu)}\di n)\o a_{(1,\nu)}a'_{(1,\nu)}h+a'_{(0,\mu)}\cdot(a_{(0,\mu)}\di n)\o a_{(1,\nu)}a'_{(1,\nu)}h\\
&=a\cdot(a'\di (n\o h))+a'\cdot(a\di (n\o h)).
\end{align*}
Thus $N\o H$ is a Poisson $(A,H)$-Hopf module. The proof is completed.
\end{proof}

\begin{remark}
  By the above Lemma, if $H$ is commutative, then $A\o H$ is a Poisson $(A,H)$-Hopf module with the following actions and coaction
  \begin{align*}
& (a\o h)_{(0,\mu')}\o  (a\o h)_{(1,\nu')} =a\o h_{(1,\mu^{-1}\mu')}\o h_{(2,\nu')},  \\
&a\di(a_0\o h)=\{a_{(0,\mu)},a_0\}\o a_{(1,\nu)}h,\\
&a\cdot(a_0\o h)=a_{(0,\mu)}\cdot a_0\o a_{(1,\nu)}h,
  \end{align*}
for all $a\in A_{\a},a\o h\in A_\mu\o H_\nu$ with $\mu\nu=\mu'\nu'=\a$.
\end{remark}

Now denote $\pi_{\mu,\nu}:(N\o H)_\a\rightarrow N_\mu\o H_\nu$ the canonical projection with $\mu\nu=\a$.

\begin{lemma}\label{Lem:2g}
(i) Let $M$ be an $(\underline{A},H)$-comodule and $N$ a Poisson $A$-module. There exists a linear isomorphism
\begin{align*}
\gamma:&\quad\!_{\underline{A}}Hom^H(M,N\o H)\rightarrow\! _{\underline{A}}Hom(M,N),\\
&f=\{f_\a\}_{\a\in G}\mapsto\{(id\o\varepsilon)\circ \pi_{\a,e}\circ f_\a\}_{\a\in G}.
\end{align*}

(ii) Let $H$ be commutative, $M$ a Poisson $(A, H)$-Hopf module and $N$ a Poisson $A$-module. There exists a linear isomorphism
\begin{align*}
\gamma:&\quad\!_{\underline{A}}Hom^H(M,N\o H)\rightarrow\! _{\underline{A}}Hom(M,N),\\
&f=\{f_\a\}_{\a\in G}\mapsto\{(id\o\varepsilon)\circ \pi_{\a,e}\circ f_\a\}_{\a\in G}.
\end{align*}
\end{lemma}

\begin{proof}
(1) For $f\in\!_{\underline{A}}Hom^H(M,N\o H)$ and $(m_\a)\in M$, set 
$$f_\a(m_\a)=\prod_{\mu\nu=\a}\left(\sum_{i}n^\mu_i\o h^\nu_i\right),$$
for all $m_\a\in M_\a,n^\mu_i\in N_\mu, h^\nu_i\in  H_\nu$. Then 
$$\g(f)_\a(m_\a)=\sum_{i}n^\a_i\varepsilon(h^e_i).$$
For $a\in A_\a,$ we have
$$f_\a(a\di m_\a)=a\di f_\a(m_\a)=\sum_{\mu\nu=\a}\sum_{i}a_{(0,\mu)}\di n^\mu_i\o a_{(1,\nu)}h^\nu_i.$$
Since
\begin{align*}
\g(f)_\a(a\di m_\a)&=\sum_{i}a_{(0,\a)}\di n^\a_i\varepsilon(a_{(1,e)}h^e_i)\\
&=\sum_{i}a\di n^\a_i\varepsilon(h^e_i)=a\di (\sum_{i}n^\a_i\varepsilon(h^e_i))\\
&=a\di \g(f)_\a(m_\a),
\end{align*}
$\g(f)$ is a Lie $A$-linear. For $g\in\! _{\underline{A}}Hom(M,N)$, define $\overline{g}:M\rightarrow N\o H$ by
$$\overline{g}_\a=\prod_{\mu\nu=\a}((g_\mu\o id_\nu)\circ\r_{\mu,\nu}).$$
Since for all $a\in A_\a,m\in M_{\a\b}$,
\begin{align*}
&\overline{g}_{\a\b}(m)_{(0,\a)}\o\overline{g}_{\a\b}(m)_{(1,\b)}\\
&=\prod_{\mu\nu=\a\b}((g_{\mu}(m_{(0,\mu)})\o m_{(1,\nu)})_{(0,\a)}\o(g_{\mu}(m_{(0,\mu)})\o m_{(1,\nu)})_{(1,\b)})\\
&=\prod_{\mu\nu=\a\b}(g_{\mu}(m_{(0,\mu)})\o  (m_{(1,\nu)})_{(1,\mu^{-1}\a)}\o (m_{(1,\nu)})_{(2,\b)})\\
&=\prod_{\mu\nu=\a\b}(g_{\mu}(m_{(0,\mu)})\o  m_{(1,\mu^{-1}\a)}\o m_{(2,\b)})\\
&=\prod_{\mu\nu=\a}(g_{\mu}(m_{(0,\mu)})\o  m_{(1,\nu)}\o m_{(2,\b)})\\
&=\overline{g}_{\a}(m_{(0,\a)})\o m_{(1,\b)},
\end{align*}
and 
\begin{align*}
\overline{g}_{\a}(a\di m_\a)&=\prod_{\mu\nu=\a}(g_\mu(a\di m_{\a})_{(0,\mu)}\o (a\di m_{\a})_{(1,\nu)})\\
&=\prod_{\mu\nu=\a}(g_\mu(a_{(0,\mu)}\di m_{\a(0,\mu)})\o a_{(1,\nu)} m_{\a(1,\nu)})\\
&=\prod_{\mu\nu=\a}(a_{(0,\mu)}\di g_\mu(m_{\a(0,\mu)})\o a_{(1,\nu)} m_{\a(1,\nu)})\\
&=a\di\prod_{\mu\nu=\a}(g_\mu(m_{\a(0,\mu)})\o m_{\a(1,\nu)})\\
&=a\di \overline{g}_{\a}(m_\a),
\end{align*}
we obtain that $\overline{g}$ is a morphism in $_{\underline{A}}Hom^H(M,N\o H)$. We denote $\gamma'(g)=\overline{g}$, and we have a morphism $\g': \! _{\underline{A}}Hom(M,N)\rightarrow\!_{\underline{A}}Hom^H(M,N\o H)$.  For all $g\in \! _{\underline{A}}Hom(M,N)$,
\begin{align*}
\g(\g'(g))_\a(m_\a)&=[(id\o\varepsilon)\circ \pi_{\a,e}\circ \g'(g)_\a](m_\a)\\
&=[(id\o\varepsilon)\circ \pi_{\a,e}]\left(\prod_{\mu\nu=\a}(g_\mu(m_{\a(0,\mu)})\o  m_{\a(1,\nu)})\right)\\
&=(id\o\varepsilon)(g_\a(m_{\a(0,\a)})\o  m_{\a(1,e)})\\
&=g_\a(m_\a),
\end{align*}
so $\g\circ\g'=id_{ _{\underline{A}}Hom(M,N)}$. Note that if $f\in\!_{\underline{A}}Hom^H(M,N\o H),$ for all $m_\a\in M_\a$,
$$f_\a(m_\a)_{(0,\mu')}\o f_\a(m_\a)_{(1,\nu')}=\prod_{\mu\nu=\a}\left(\sum_{i}n^\mu_i\o h^{\nu}_{i(1,\mu^{-1}\mu')}\o h^{\nu}_{i(2,\nu')}\right).$$
then
\begin{align*}
\g'(\g(f))_\a(m_\a)&=\prod_{\mu\nu=\a}\left(\g(f)_\mu(m_{\a(0,\mu)})\o m_{\a(1,\nu)}\right)\\
&=\prod_{\mu\nu=\a}\left(\g(f)_\mu(m_\a)_{(0,\mu)}\o \g(f)_\mu(m_\a)_{(1,\nu)}\right)\\
&=\prod_{\mu\nu=\a}(id\o\varepsilon)\left(\sum_{i}n^\mu_i\o h^{\nu}_{i(1,e)}\right)\o h^{\nu}_{i(2,\nu)}\\
&=\prod_{\mu\nu=\a}\left(\sum_{i}n^\mu_i\o h^{\nu}_{i}\right)\\
&=f_\a(m_\a),
\end{align*}
so $\g'\circ\g=id_{ _{\underline{A}}Hom^H(M,N\o H)}$. The proof is completed.
\end{proof}

By Lemma \ref{Lem:2g}, we could directly get

\begin{corollary}\label{coro:2h}
Let $H$ be a Hopf $G$-algebra and $A$  an $H$-comodule Poisson algebra.

 (1) If $N$ is an injective Poisson $A$-module, then $N\o H$ is an injective $(\underline{A}, H)$-comodule.

(2) Let $H$ commutative. If $N$ is an injective Poisson $A$-module, then $N\o H$ is an injective Poisson $(A, H)$-Hopf module.
\end{corollary}

\begin{theorem}\label{the:2i}
Let $H$ be a Hopf $G$-algebra and $A$  an $H$-comodule Poisson algebra. Suppose that there is an $H$-colinear map $\p:H\rightarrow A^A$ such that for all $\a\in G$, $\p_\a(1_{H_\a})=1_{A_\a}$.

(1) If $H$ is commutative or  $\phi$ is an algebra map, then every Poisson $(A, H)$-Hopf module which is injective as a Lie $A$-module is an injective $(\underline{A}, H)$-comodule.

(2) If $H$ is commutative and $\phi$ is an algebra map, then every Poisson $(A, H)$-Hopf module which is injective as a Poisson $A$-module is an injective Poisson $(A, H)$-Hopf module.
\end{theorem}

\begin{proof}
(1) Let $M$ be a Poisson $(A,H)$-Hopf module and an injective Lie $A$-module. Consider the $k$-linear map $\l:M\o H\rightarrow M$ given by
$$\l_\a((m_\mu\o h_\nu)_{\mu\nu=\a})=\p_\a(h_\a S^{-1}_\a(m_{e(1,\a^{-1})}))\cdot m_{e(0,\a)},$$
where $(m_\mu\o h_\nu)_{\mu\nu=\a}\in (M\o H)_\a.$ Let $\r^M_\a=\prod\limits_{\mu\nu=\a}\r_{\mu,\nu}$, then for all $m_\a\in M_\a$,
\begin{align*}
&(\l_\a\circ\r^M_\a)(m_\a)=\l_\a((m_{(0,\mu)}\o m_{(1,\nu)})_{\mu\nu=\a})\\
&=\p_\a(m_{\a(1,\a)}S^{-1}_\a(m_{\a(0,e)(1,\a^{-1})})\cdot m_{\a(0,e)(0,\a)}\\
&=\p_\a(m_{\a(2,\a)}S^{-1}_\a(m_{\a(1,\a^{-1})})\cdot m_{\a(0,\a)}\\
&=\p_\a(1_{H_\a})\cdot m_\a=m_\a,
\end{align*}
thus $\l\circ \r^M=id_M$, i.e., $\r^M$ is an injective map. Let $m\o h=(m_\mu\o h_\nu)_{\mu\nu=\a}\in (M\o H)_\a$,  for $\mu'\nu'=\a$, we have
\begin{align*}
&\l_\a(m\o h)_{(0,\mu')}\o\l_\a(m\o h)_{(1,\nu')}\\
&=\p_\a(h_\a S^{-1}_\a(m_{e(2,\a^{-1})}))_{(0,\mu')}\cdot m_{e(0,\mu')}\o \p_\a(h_\a S^{-1}_\a(m_{e(2,\a^{-1})}))_{(1,\nu')}m_{e(1,\nu')}\\
&=\p_{\mu'}(h_{\a(1,\mu')} S^{-1}_{\mu'}(m_{e(3,\mu'^{-1})}))\cdot m_{e(0,\mu')}\o h_{\a(2,\nu')} S^{-1}_{\nu'}(m_{e(2,\nu'^{-1})})m_{e(1,\nu')}\\
&=\p_{\mu'}(h_{\a(1,\mu')} S^{-1}_{\mu'}(m_{e(1,\mu'^{-1})}))\cdot m_{e(0,\mu')}\o h_{\a(2,\nu')}\\
&=(\l_{\mu'}\o id)(((m_\mu\o h_{\nu(1,\mu^{-1}\mu')}))\o h_{\nu(2,\nu')} )  \\
&=\l_{\mu'}((m\o h)_{(0,\mu')})\o (m\o h)_{(1,\nu')},
\end{align*}
that is, $\l$ is $H$-colinear. And on one hand,
\begin{align*}
\l_\a(a\di (m\o h))&=\l_\a((a_{(0,\mu)}\di m_\mu\o a_{(1,\nu)}h_\nu))\\
&=\p_\a(a_{(1,\a)}h_\a S^{-1}_\a((a_{(0,e)}\di m_e)_{(1,\a^{-1})}))\cdot (a_{(0,e)}\di m_e)_{(0,\a)}\\
&=\p_\a(a_{(2,\a)}h_\a S^{-1}_\a((a_{(1,\a^{-1})}m_{e(1,\a^{-1})}))\cdot (a_{(0,\a)}\di m_{(0,\a)}).
\end{align*}
On the other hand,
\begin{align*}
&a\di\l_\a(m\o h)\\
&=a\di [\p_\a(h_\a S^{-1}_\a(m_{e(1,\a^{-1})}))\cdot m_{e(0,\a)}]\\
&=\{a\di \p_\a(h_\a S^{-1}_\a(m_{e(1,\a^{-1})}))\}\cdot m_{e(0,\a)}+\p_\a(h_\a S^{-1}_\a(m_{e(1,\a^{-1})}))\cdot(a\di m_{e(0,\a)})\\
&=\p_\a(h_\a S^{-1}_\a(m_{e(1,\a^{-1})}))\cdot(a\di m_{e(0,\a)}),
\end{align*}
where the third identity holds since $\p_\a(h_\a S^{-1}_\a(m_{e(1,\a^{-1})}))\in A^{A_\a}_\a$. Now when $H$ is commutative or $\p$ is an algebra map, we obtain
$$\l_\a(a\di (m\o h))=a\di\l_\a(m\o h),$$
thus $\l$ is Lie $A$-linear. Therefore $\l$ is a morphism in $_{\underline{A}}\mathcal{M}^H$. By \ref{coro:2h}, $M\o H$ is an injective $(\underline{A},H)$-comodule. It is straightforward to check that the comudule structure map $\r^M$ is also a morphism in $_{\underline{A}}\mathcal{M}^H$. So $M$ is a direct sum of $M\o H$ as an object in $_{\underline{A}}\mathcal{M}^H$, which implies that $M$ is an injective $(\underline{A},H)$-comodule. 

(2) By a similar computations, we obtain that 
$$\l_\a(a\cdot (m\o h))=\p_\a(a_{(2,\a)}h_\a S^{-1}_\a((a_{(1,\a^{-1})}m_{e(1,\a^{-1})}))\cdot (a_{(0,\a)}\cdot m_{(0,\a)}),$$
and
$$a\cdot\l_\a(m\o h)=\p_\a(h_\a S^{-1}_\a(m_{e(1,\a^{-1})}))\cdot(a\cdot m_{e(0,\a)}).$$
Also by a similar analysis, we obtain that $M$ is an injective Poisson $(A,H)$-Hopf module. 
\end{proof}

\section{Fundamental theorem of Poisson Hopf modules}
\def\theequation{3.\arabic{equation}}
\setcounter{equation} {0}

In what follows, we will always denote $B=A^{AcoH}$. 

\begin{lemma}\label{lem:2j}
Let $M$ be a Poisson $(A,H)$-Hopf module. Then 
$$A\o_B M^{AcoH}=\{A_\a\o_{B_\a} M^{AcoH}_\a\}_{\a\in G}$$ 
is a Poisson $(A,H)$-Hopf module with the following structures:
\begin{align*}
&a'\cdot (a\o_{B_\a} m_\a)=a'a\o_{B_\a} m_\a,\\
&a'\di(a\o_{B_\a} m_\a)=\{a',a\}\o_{B_\a} m_\a,\\
&(a\o_{B_\a} m_{\a})_{(0,\mu)}\o(a\o m_\a)_{(1,\nu)}=a_{(0,\mu)}\o_{B_\mu} m_\mu\o a_{(1,\nu)},
\end{align*}
for all $\a,\mu,\nu\in G,\mu\nu=\a,$ and $a,a'\in A_\a, m=(m_\a)\in M^{AcoH}$.
\end{lemma}

\begin{proof}
By a similar verification as in \cite{Li}, $A\o_B M^{AcoH}$ is an $(A,H)$-Hopf module. For all $a,a',a''\in A_\a,b\in B_\a$ and $m\in M^{AcoH}$,
\begin{align*}
a'\di (ab\o_{B_\a} m_\a)&=\{a',ab\}\o_{B_\a} m_\a\\
&=[\{a',a\}b+a\{a',b\}]\o_{B_\a} m_\a\\
&=\{a',a\}b\o_{B_\a} m_\a\\
&=\{a',a\}\o_{B_\a} b\cdot m_\a\\
&=a'\di(a\o_{B_\a} b\cdot m_\a),
\end{align*}
so the Lie $A$-action on $A\o_B M^{AcoH}$ is well defined. 
\begin{align*}
a''\di[a'\cdot(a\o_{B_\a} m_\a)]&=a''\di(a'a\o_{B_\a} m_\a)\\
&=\{a'',a'a\}\o_{B_\a} m_\a\\
&=[\{a'',a\}a'+a\{a'',a'\}]\o_{B_\a} m_\a\\
&=a'\cdot[a''\di(a\o_{B_\a} m_\a)]+\{a'',a'\}\cdot(a\o_{B_\a} m_\a),
\end{align*}
and 
\begin{align*}
a'a''\di(a\o_{B_\a} m_\a)&=\{a'a'',a\}\o_{B_\a} m_\a\\
&=[\{a',a\}a''+a'\{a'',a\}]\o_{B_\a} m_\a\\
&=a''\cdot(\{a',a\}\o_{B_\a} m_\a)+a'\cdot(\{a'',a\}\o_{B_\a} m_\a)\\
&=a''\cdot[a'\di (a\o_{B_\a} m_\a)]+a'\cdot[a''\di(a\o_{B_\a} m_\a)],
\end{align*}
hence $A\o_B M^{AcoH}$ is a Poisson $A$-module. Moreover
\begin{align*}
&(a'\di(a\o_{B_\a} m_\a))_{(0,\mu)}\o(a'\di(a\o_{B_\a} m_\a))_{(1,\nu)}\\
&=(\{a',a\}\o_{B_\a} m_\a)_{(0,\mu)}\o(\{a',a\}\o_{B_\a} m_\a)_{(1,\nu)}\\
&=\{a',a\}_{(0,\mu)}\o_{B_\mu} m_\mu\o \{a',a\}_{(1,\nu)}\\
&=\{a'_{(0,\mu)},a_{(0,\mu)}\}\o_{B_\mu} m_\mu\o a'_{(1,\nu)}a_{(1,\nu)}\\
&=a'_{(0,\mu)}\di (a\o_{B_\a} m_\a)_{(0,\mu)}\o a'_{(1,\nu)}(a\o_{B_\a} m_\a)_{(1,\nu)}.
\end{align*}
Therefore $A\o_B M^{AcoH}$ is a Poisson $(A,H)$-Hopf module. The proof is completed.
\end{proof}

\begin{lemma}\label{lem:2k}
Let  $M$ be a Poisson $(A,H)$-Hopf module. Then the $k$-linear map $\Phi:A\o_B M^{AcoH}\rightarrow M$ given by
$$\Phi_\a(a\o m_\a)=a\cdot m_\a,$$
is a homomorphism of Poisson $(A,H)$-Hopf modules, where $a\in A_\a$ and $m=(m_\a)\in M^{AcoH}$.
\end{lemma}

\begin{proof}
Clearly $\Phi$ is a map of $A$-module. For all $a,a'\in A_\a$ and $m=(m_\a)\in M^{AcoH}$,
\begin{align*}
&\Phi_\a(a'\di(a\o_{B_\a} m_\a))\\
&=\Phi_\a(\{a',a\}\o_{B_\a} m_\a)\\
&=\{a',a\}\cdot m_\a\\
&=a'\di\Phi_\a(a\o_{B_\a} m_\a),
\end{align*}
and 
\begin{align*}
&\Phi_\a(a\o_{B_\a} m_\a)_{(0,\mu)}\o \Phi_\a(a\o_{B_\a} m_\a)_{(1,\nu)}\\
&=(a\cdot m_\a)_{(0,\mu)}\o(a\cdot m_\a)_{(1,\nu)}\\
&=a_{(0,\mu)}\cdot m_{\a(0,\mu)}\o a_{(1,\nu)} m_{\a(1,\nu)}\\
&=a_{(0,\mu)}\cdot m_\mu\o a_{(1,\nu)}\\
&=\Phi_\mu((a\o_{B_\a} m_\a)_{(0,\mu)})\o (a\o_{B_\a} m_\a)_{(1,\nu)},
\end{align*}
which means that $\Phi$ is a morphism of $H$-comodules.
\end{proof}

Assume that there exists a map of right $H$-comodules $\phi:H\rightarrow A$ such that $\p_\a(1_{H_\a})=1_{A_\a}$. For any  object $M$ in $_{\mathcal{P}_A}\mathcal{M}^H$, consider the $k$-linear map
$$p^M_\a:M_e\rightarrow M_\a,\ p^M_\a(m)=\p_\a(S^{-1}_\a(m_{(1,\a^{-1})}))\cdot m_{(0,\a)},$$
for all $m\in M_e$.

\begin{lemma}\label{lem:2l}
Let $p^M_\a$ defined as above. Then we have $\left\{(p^M_\a(m))_{\a\in G}|m\in M_e\right\}= M^{coH}$.
\end{lemma}

\begin{proof}
The verification is straightforward and left to the reader.
%
\end{proof}

Let $M\in\!_{\mathcal{P}_A}\mathcal{M}^H$. Define the action of Lie algebra $A_e$ on $M^{coH}_\a$ by
$$a\di' p^M_\a(m)=p^M_\a(a\di m),$$
for all $a\in A_e,m\in M_e$.

\begin{lemma}\label{lem:2m}
\begin{itemize}
  \item [(1)]  For all $a\in A_e, m\in M_e$, 
  $$p^M_\a(a\cdot m)=p^A_\a(a)\cdot p^M_\a(m),\ \ p^M_\a(a\di p^M_e(m))=p^M_\a(a\di m).$$
  \item [(2)] For all $a\in A_\a, m\in M_e$,  
  $$a\di p^M_\a( m)=\p_\a(a_{(1,\a)})\cdot(a_{(0,e)}\di'p^M_\a( m)).$$
  \item [(3)] For all $a,c\in A_e, m\in M_e$, 
  $$\{a,c\}\di' p^M_\a( m)=a\di'(c\di' p^M_\a( m))-c\di'(a\di'p^M_\a( m)).$$
  \item [(4)] For all $m\in M_\a$,   
  $$\p_\a(m_{(1,\a)})\cdot\p^M_\a(m_{(0,e)})=m.$$
\end{itemize}
\end{lemma}

\begin{proof}
(1) For all $a\in A_e, m\in M_e$, 
\begin{align*}
p^M_\a(a\cdot m)& =\p_\a(S^{-1}_\a((a\cdot m)_{(1,\a^{-1})}))\cdot (a\cdot m)_{(0,\a)} \\
& =\p_\a(S^{-1}_\a(a_{(1,\a^{-1})} m_{(1,\a^{-1})}))a_{(0,\a)}\cdot m_{(0,\a)} \\
&=\p_\a(S^{-1}_\a(m_{(1,\a^{-1})})S^{-1}_\a(a_{(1,\a^{-1})} ))a_{(0,\a)}\cdot m_{(0,\a)} \\
&=\p_\a(S^{-1}_\a(a_{(1,\a^{-1})} ))a_{(0,\a)}\p_\a(S^{-1}_\a(m_{(1,\a^{-1})}))\cdot m_{(0,\a)} \\
&=p^A_\a(a)\cdot p^M_\a(m),
\end{align*}
and 
\begin{align*}
&p^M_\a(a\di p^M_e(m))\\
&=p^M_\a(a\di(\p_e(S^{-1}_e(m_{(1,e)}))\cdot m_{(0,e)}))\\
&=p^M_\a\left((a\di \p_e(S^{-1}_e(m_{(1,e)})))\cdot m_{(0,e)}+\p_e(S^{-1}_e(m_{(1,e)}))\cdot (a\di m_{(0,e)})\right)\\
&=p^M_\a\left(\p_e(S^{-1}_e(m_{(1,e)}))\cdot (a\di m_{(0,e)})\right)\\
&=p^A_\a(\p_e(S^{-1}_e(m_{(1,e)})))\cdot p^M_\a(a\di m_{(0,e)})\\
&=\p_\a(S^{-1}_\a(\p_e(S^{-1}_e(m_{(1,e)})))_{(1,\a^{-1})})\p_e(S^{-1}_e(m_{(1,e)}))_{(0,\a)}\\
&\quad\p_\a(S^{-1}_\a(a\di m_{(0,e)})_{(1,\a^{-1})})\cdot  (a\di m_{(0,e)})_{(0,\a)}\\
&=\p_\a(S^{-1}_\a(S^{-1}_e(m_{(2,e)})_{(2,\a^{-1})}))\p_\a(S^{-1}_e(m_{(2,e)})_{(1,\a)})\\
&\quad\p_\a(S^{-1}_\a(a_{(1,\a^{-1})} m_{(1,\a^{-1})}))\cdot  (a_{(0,\a)}\di m_{(0,\a)})\\
&=\p_\a(S^{-1}_\a(S^{-1}_{\a^{-1}}(m_{(2,\a)})))\p_\a(S^{-1}_\a(m_{(3,\a^{-1})}))\p_\a(S^{-1}_\a(a_{(1,\a^{-1})} m_{(1,\a^{-1})}))\cdot  (a_{(0,\a)}\di m_{(0,\a)})\\
&=\p_\a(S^{-1}_\a(a_{(1,\a^{-1})} m_{(1,\a^{-1})}))\cdot  (a_{(0,\a)}\di m_{(0,\a)})\\
&=p^M_\a(a\di m).
\end{align*}

(2) For all $a\in A_\a, m\in M_e$,
\begin{align*}
&\p_\a(a_{(1,\a)})\cdot(a_{(0,e)}\di'p^M_\a( m))\\
&=\p_\a(a_{(1,\a)})\cdot p^M_\a(a_{(0,e)}\di p^M_\a( m))\\
&=\p_\a(a_{(1,\a)}) \p_\a(S^{-1}_\a((a_{(0,e)}\di p^M_\a( m))_{(1,\a^{-1})}))\cdot (a_{(0,e)}\di p^M_\a( m))_{(0,\a)}\\
&=\p_\a(a_{(2,\a)}) \p_\a(S^{-1}_\a(a_{(1,\a^{-1})}p^M_\a(m)_{(1,\a^{-1})}))\cdot (a_{(0,\a)}\di p^M_\a(m)_{(0,\a)})\\
&=\p_\a(a_{(2,\a)}) \p_\a(S^{-1}_\a(a_{(1,\a^{-1})}))\cdot (a_{(0,\a)}\di p^M_\a(m))\\
&=a\di p^M_\a(m),
\end{align*}
where the  third identity  uses Lemma \ref{lem:2l}.

(3) For all $a,c\in A_e, m\in M_e$, 
\begin{align*}
\{a,c\}\di' p^M_\a(m)&=p^M_\a(\{a,c\}\di m)\\
&=p^M_\a(a\di (c\di m)-c\di(a\di m))\\
&=p^M_\a(a\di (c\di m))-p^M_\a(c\di(a\di m))\\
&=a\di' p^M_\a(c\di m)-c\di' p^M_\a(a\di m)\\
&=a\di' (c\di' p^M_\a(m))-c\di' (a\di' p^M_\a(m)).
\end{align*}

(4) The verification is straightforward.
\end{proof}

\begin{remark}
For all $a,c\in A_e, m\in M_e$, since
\begin{align*}
&(a\di' p^M_{\a\b}(m))_{(0,\a)}\o (a\di' p^M_{\a\b}(m))_{(1,\b)}\\
&=p^M_{\a\b}(a\di m)_{(0,\a)}\o p^M_{\a\b}(a\di m)_{(1,\b)}\\
&=p^M_{\a}(a\di m)\o1_{H_\b}=a\di' p^M_{\a}(m)\o1_{H_\b},
\end{align*}
it follows from Lemma \ref{lem:2m} (3) that for every Poisson $(A,H)$-Hopf module, $M^{coH}$ is a Lie $A_e$-module under the action $\di'$. Therefore $A^{coH}$  is also a Lie $A_e$-module under the action $\di'$. 
\end{remark}

\begin{lemma}\label{lem:2n}
Let  $M$ be a Poisson $(A,H)$-Hopf module. Assume that there exists a map of right $H$-comodules $\phi:H\rightarrow A$ which is also an algebra map.
\begin{itemize}
  \item [(1)] If the action $\di'$ of $A_e$ on $M^{coH}$ is trivial, then $M^{AcoH}=M^{coH}$.
  \item [(2)] If the action $\di'$ of $A_e$ on $A^{coH}$ is trivial, then $A^{AcoH}=A^{coH}$.
\end{itemize}
\end{lemma}

\begin{proof}
For all $a\in A_\a, m\in M_e$, by Lemma \ref{lem:2m} (2),
$$a\di p^M_\a(m)=\p_\a(a_{(1,\a)})\cdot(a_{(0,e)}\di'p^M_\a( m))=0,$$
thus $p^M_\a(m)\in M^{AcoH}$ and $M^{coH}\subseteq M^{AcoH}$, as required.
\end{proof}

We are now in the position to give the second main result: fundamental theorem for Poisson $(A,H)$-Hopf modules. 

\begin{theorem}\label{thm:2o}
 Let $M$ be a Poisson $(A, H)$-Hopf module, and $\phi:H\rightarrow A^A$ a right $H$-colinear algebra map. Suppose $M^{coH}$ and $A^{coH}$ are trivial Lie $A_e$-modules under $\di'$. Then the morphism $\Phi$ given in Lemma \ref{lem:2k} is an isomorphism of Poisson $(A, H)$-Hopf modules.
\end{theorem}

\begin{proof}
We have already known that $\Phi$ is a homomorphism of Poisson $(A, H)$-Hopf modules. By the assumption, the action $\di'$ of $A_e$ on $M^{coH}$ is trivial, then we have $M^{AcoH}=M^{coH}$. Hence we obtain a family of well-defined $k$-linear maps 
$$\Psi_\a:M_\a\rightarrow A_\a\o_{B_\a}(M^{AcoH})_\a,\ m\mapsto \phi_\a(m_{(1,\a)})\o_{B_\a}p^M_\a(m_{(0,e)}).$$
Let $\Psi=\{\Psi_\a\}_{\a\in G}$. For all $a\in A_\a,m=(m_\a)\in M^{AcoH}$,
\begin{align*}
&(\Psi_\a\circ\Phi_\a)(a\o_{B_\a}m_\a)=\Psi_\a(a\cdot m_\a)\\
&=\phi_\a((a\cdot m_\a)_{(1,\a)})\o_{B_\a}p^M_\a((a\cdot m_\a)_{(0,e)})\\
&=\phi_\a(a_{(1,\a)} m_{_\a(1,\a)})\o_{B_\a}p^M_\a(a_{(0,e)}\cdot m_{_\a(0,e)})\\
&=\phi_\a(a_{(1,\a)} 1_{\a})\o_{B_\a}p^A_\a(a_{(0,e)})\cdot p^M_\a(m_{e})\\
&=\phi_\a(a_{(1,\a)})p^A_\a(a_{(0,e)})\o_{B_\a}p^M_\a(m_e)\\
&=a\o_{B_\a}m_\a.
\end{align*}
Thus $\Psi\circ\Phi=id_{A\o_{B}M^{AcoH}}$. It is a routine exercise to verify that $\Phi\circ\Psi=id_M$. The proof is completed.
\end{proof}

\begin{definition}
An  $H$-comodule $N=\{N_\a\}_{\a\in G}$ is called trivial if $N=N^{coH}$, namely for all $\a,\b\in G,n_{\a\b}\in N_{\a\b}$, there exists an unique element $n_\a\in N_\a$ such that 
$$\r_{\a,\b}(n_{\a\b})=n_\a\o 1_{H_\b}.$$
\end{definition}

\begin{lemma}\label{lem:2p}
Let $N$ be a left $B$-module and a trivial right $H$-comodule. Then $A\o_BN=\{A_\a\o_{B_\a}N_\a\}_{\a\in G}$ is a Poisson $(A,H)$-Hopf module with the following structure
\begin{align*}
&a'\cdot (a\o_{B_\a} n'_\a)=a'a\o_{B_\a}n'_\a ,\\
&a'\di(a\o_{B_\a} n'_\a)=\{a',a\}\o_{B_\a} n'_\a,\\
&(b\o_{B_\a} n_{\a\b})_{(0,\a)}\o (b\o_{B_\a} n_{\a\b})_{(1,\b)}=b_{(0,\a)}\o_Bn_\a\o b_{(1,\b)},
\end{align*}
for all $a,a'\in A_{\a},b\in  A_{\a\b},n'_\a\in N_\a,n_{\a\b}\in N_{\a\b}$.
\end{lemma}

\begin{proof}
The proof is similar to Lemma \ref{lem:2j}.
\end{proof}

For Poisson $(A,H)$-Hopf modules $M,N$ and $f:M\rightarrow N$ in $_{\mathcal{P}A}\mathcal{M}^H$, since for $m=(m_\a)\in M^{AcoH},a\in A_\a$, 
\begin{align*}
f_{\a\b}(m_{\a\b})_{(0,\a)}\o f_{\a\b}(m_{\a\b})_{(1,\b)}&=f_{\a}(m_{\a\b(0,\a)})\o m_{\a\b(1,\b)}\\
&=f_{\a}(m_{\a})\o 1_{H_\b},
\end{align*}
and 
$$a\di f_\a(m_\a)=f(a\di m_\a)=0,$$
we have $f(m)=(f_\a(m_\a))\in N^{AcoH}$. This gives a functor
$$F_1=(-)^{AcoH}:\!_{\mathcal{P}A}\mathcal{M}^H\rightarrow\!_B\mathcal{M},\ M\mapsto M^{AcoH}.$$
In what follows, we will denote the subcategory of $_B\mathcal{M}$ whose objects has a  trivial $H$-comodule structure by $_B\mathcal{M}^{H_0}$. By Lemma \ref{lem:2p}, we also have a functor
$$F_2:\!_B\mathcal{M}^{H_0}\rightarrow\!_{\mathcal{P}A}\mathcal{M}^H, \ N\mapsto A\o_B N.$$

We denote the set of all morphisms of both $B$-module and trivial $H$-comodule by $Hom^{H_0}_B(N,M^{AcoH})$.
\begin{proposition}\label{pro:2q}
Let $M\in\!_{\mathcal{P}A}\mathcal{M}^H$ and $N\in\!_B\mathcal{M}^{H_0}$. There is a functorial isomorphism 
$$
\psi:\!_{\mathcal{P}A}Hom^H(A\o_BN,M)\rightarrow Hom^{H_0}_B(N,M^{AcoH}),
$$
where for any $f\in\!_{\mathcal{P}A}Hom^H(A\o_BN,M),n\in N_\a$,  
$\psi(f)_\a(n)=f_\a(1_{A_\a}\o n).$
\end{proposition}

\begin{proof}
Since for all $n_{\a\b}\in N_{\a\b}$, 
\begin{align*}
&(\psi(f)_{\a\b}(n_{\a\b}))_{(0,\a)}\o(\psi(f)_{\a\b}(n_{\a\b}))_{(1,\b)}\\
&=f_{\a\b}(1\o n_{\a\b})_{(0,\a)}\o f_{\a\b}(1\o n_{\a\b})_{(1,\b)}\\
&=f_{\a}((1\o n_{\a\b})_{(0,\a)})\o (1\o n_{\a\b})_{(1,\b)}\\
&=f_{\a}(1_{A_\a}\o n_{\a})\o 1_{H_\b}\\
&=\psi(f)_\a(n_\a)\o 1_{H_\b},
\end{align*}
we have Im$(\psi(f))\subseteq M^{AcoH}$, and that $\psi(f)$ is $H$-colinear. Easy to see that $\psi(f)_\a(n)\in M^{A_\a}_\a$ for all $n\in N_\a$ and $\psi(f)$ is an $B$-module map, thus 
the map $\psi$ is well-defined, where $\r_{\a,\b}(n_{\a\b})=n_\a\o 1_{H_\b}.$

Define $\psi':Hom^{H_0}_B(N,M^{AcoH})\rightarrow\!_{\mathcal{P}A}Hom^H(A\o_BN,M) $ by
$$\psi'(g)_\a(a\o_{B_\a}n)=a\cdot g_\a(n).$$
For all $a,a'\in A_\a,n\in N_\a$,
\begin{align*}
&\psi'(g)_\a(a'\di(a\o_{B_\a}n))=\psi'(g)_\a(\{a',a\}\o_{B_\a}n)\\
&=\{a',a\}\cdot g_\a(n)=a'\di(a\cdot m)-a\cdot(a'\di n)\\
&=a'\di(a\cdot m)=a'\di \psi'(g)_\a(a\o_{B_\a}n),
\end{align*}
i.e., $\psi'(g)$ is Lie $A$-linear. For $b\in A_{\a\b},n_{\a\b}\in N_{\a\b}$,
\begin{align*}
&\psi'(g)_{\a\b}(b\o_{B_{\a\b}}n_{\a\b})_{(0,\a)}\o \psi'(g)_{\a\b}(b\o_{B_\a}n_{\a\b})_{(1,\b)}\\
&=(b\cdot g_{\a\b} (n_{\a\b}))_{(0,\a)}\o (b\cdot g_{\a\b} (n_{\a\b}))_{(1,\b)}\\
&=b_{(0,\a)}\cdot g_{\a\b} (n_{\a\b})_{(0,\a)}\o b_{(1,\b)} g_{\a\b} (n_{\a\b})_{(1,\b)}\\
&=b_{(0,\a)}\cdot g_{\a\b} (n_{\a\b(0,\a)})\o b_{(1,\b)} n_{\a\b(1,\b)}\\
&=b_{(0,\a)}\cdot g_{\a} (n_{\a})\o b_{(1,\b)} \\
&=\psi'(g)_\a(b_{(0,\a)}\o_{B_\a}n_{\a})\o b_{(1,\b)}\\
&=\psi'(g)_\a((b\o_{B_{\a\b}}n_{\a\b})_{(0,\a)})\o (b\o_{B_{\a\b}}n_{\a\b})_{(1,\b)}.
\end{align*}
Therefore $\psi'$ is well-defined. It is a routine exercise to check that $\psi$ and $\psi'$ are mutual inverses. The proof is completed.
\end{proof}

\begin{corollary}\label{cor:2r}
The functors $F_1$ and $F_2$ are adjoint pair with unit and counit
$$\e_N:N\rightarrow (A\o_BN)^{AcoH},\ \e_{N_\a}(n)=1_{A_\a}\o_{B_\a}n,$$
and 
$$\d_M:A\o_BM^{AcoH}\rightarrow M,\ \d_{M_\a}(a\o_{B_{\a}}m)=a\cdot m.$$
\end{corollary}

\begin{proof}
The proof is easy to verify and omitted.
\end{proof}

\begin{corollary}
 Let $M$ be a Poisson $(A, H)$-Hopf module, and $\phi:H\rightarrow A^A$ a right $H$-colinear algebra map. Suppose $M^{coH}$ and $A^{coH}$ are trivial Lie $A_e$-modules under $\di'$. Then the functor $F_1=(-)^{AcoH}:\!_{\mathcal{P}A}\mathcal{M}^H\rightarrow\!_B\mathcal{M}^{H_0}$ is dual Maschke, that is, every object of $_{\mathcal{P}A}\mathcal{M}^H$ is $F_1$-relative projective.
\end{corollary}

\begin{proof}
The conclusion follows from Theorem \ref{thm:2o}, Proposition \ref{pro:2q},  Corollary \ref{cor:2r} and \cite[Theorem 3.4]{cae}.
\end{proof}

\section*{Acknowledgement}

This work was supported by the NSF of China (No. 12271292, 11901240).

\end{document}